\documentclass[reqno]{amsart}
\usepackage{amssymb,latexsym,amsmath,amsthm,enumerate,amsbsy}
\usepackage[mathscr]{eucal}
\usepackage{framed,color,graphicx}
\usepackage{mathrsfs}
\usepackage[all]{xy}
\usepackage{tikz}
\usepackage{cite}

\makeatletter
\@namedef{subjclassname@2020}{\textup{2020} Mathematics Subject Classification}
\makeatother

\usetikzlibrary{positioning,decorations.pathreplacing,patterns,decorations.pathmorphing}
\tikzset{%
element/.style={draw, shape=circle, fill=white, inner sep=1.4pt}
}

\DeclareSymbolFont{bbold}{U}{bbold}{m}{n}
\DeclareSymbolFontAlphabet{\mathbbold}{bbold}

\theoremstyle{plain}
\newtheorem{thm}{Theorem}[section]
\newtheorem{lem}[thm]{Lemma}
\newtheorem{cor}[thm]{Corollary}
\newtheorem{pro}[thm]{Proposition}

\theoremstyle{definition}

\newtheorem{remark}[thm]{Remark}

\newcommand{\up}[1]{\textup{#1}}

\newcommand{\bp}{\mathbf{p}}
\newcommand{\bq}{\mathbf{q}}
\newcommand{\br}{\mathbf{r}}

\newcommand{\bt}{\mathbf{t}}
\newcommand{\bu}{\mathbf{u}}
\newcommand{\bv}{\mathbf{v}}
\newcommand{\bw}{\mathbf{w}}

\begin{document}

\title[A new limit variety of additively idempotent semirings]
{A new limit variety of additively idempotent semirings}

\author{Simin Lyu}
\address{School of Mathematics, Northwest University, Xi'an, 710127, Shaanxi, P.R. China}
\email{siminlyu@yeah.net}

\author{Miaomiao Ren}
\address{School of Mathematics, Northwest University, Xi'an, 710127, Shaanxi, P.R. China}
\email{miaomiaoren@yeah.net}

\author{Mengya Yue}
\address{School of Mathematics, Northwest University, Xi'an, 710127, Shaanxi, P.R. China}
\email{myayue@yeah.net}



\begin{abstract}
We establish a sufficient condition for an additively idempotent semiring to be nonfinitely based.
Applying this condition,
we prove that the six-element additively idempotent semiring $SR_6$ has no finite basis for its identity.
Furthermore, we provide a complete description of the subvariety lattice of the variety $\mathsf{V}(SR_6)$ generated by $SR_6$,
showing that it forms a four-element chain.
Our results demonstrate that $\mathsf{V}(SR_6)$ is a limit variety:
it is itself nonfinitely based, yet all of its proper subvarieties are finitely based.
Moreover, $SR_6$ is the smallest known example of an additively idempotent semiring generating a limit variety.
\end{abstract}

\maketitle

\section{Introduction}
A \emph{variety} is a class of algebras that is closed under taking subalgebras, homomorphic images, and arbitrary direct products.
Birkhoff's celebrated theorem states that a class of algebras is a variety if and only if it is an equational class, that is,
the class of all algebras that satisfy a certain set of identities.
A variety is \emph{finitely based} if it has a finite equational basis (that is, it can be defined by a finite set of identities);
otherwise, it is \emph{nonfinitely based}.
An algebra $A$ is finitely based (resp., nonfinitely based)
if the variety $\mathsf{V}(A)$ it generates is finitely based (resp., nonfinitely based).

The \emph{finite basis problem} for a class of algebras
concerns the classification of its members according to whether they are finitely based.
This problem has been one of the central themes in universal algebra,
with deep connections to the structure of subvariety lattices and the complexity of equational theories.
It has been intensively studied for various algebraic structures, including groups, rings, semigroups, and semirings,
see, for example, \cite{aj, gk, np, vol01} and the references therein.

A variety is \textit{hereditarily finitely based} if all of its subvarieties are finitely based.
A variety is a \textit{limit variety} if it is itself nonfinitely based, but each of its proper subvarieties is finitely based.
In other words, limit varieties are precisely minimal nonfinitely based varieties.
By Zorn's lemma, every nonfinitely based variety contains a limit variety.
Consequently, a variety is hereditarily finitely based if and only if it contains no limit subvarieties.
Therefore, the problem of classifying hereditarily finitely based varieties essentially reduces to the classification of limit varieties.
As Kharlampovich and Sapir~\cite{ks} aptly observed,
the problem of finding limit varieties is interesting in itself
as is every problem about boundaries between ``Yes'' and ``No''.

Lee and Volkov~\cite{lv} noted that there is a number of challenges in classifying limit varieties,
and that even describing a single example of a limit variety can be difficult.
To date, no explicit example of a limit variety of groups has been found.
Within semigroups, however, several explicit examples are known (see~\cite{gulz, gus, j, zl}).
The aim of this paper is to present a new example of a limit variety of additively idempotent semirings.

By an \emph{additively idempotent semiring} (or ai-semiring for short) we mean
an algebra $(S, +, \cdot)$ such that the additive reduct $(S, +)$ forms a commutative idempotent semigroup,
the multiplicative reduct $(S, \cdot)$ forms a semigroup, and the distributive laws
\[
x(y+z)\approx xy+xz,\quad (x+y)z\approx xz+yz
\]
hold.
A \emph{commutative ai-semiring} is an ai-semiring whose multiplicative reduct is commutative.
Additively idempotent semirings arise naturally in various areas of mathematics and have found important applications in fields such as algebraic geometry \cite{cc}, tropical geometry \cite{ms}, information science \cite{gl} and theoretical computer science \cite{go}.

On an ai-semiring $S$, one defines a binary relation $\leq$ by
\[
a \leq b \Leftrightarrow a + b = b.
\]
This relation is a partial order and is compatible with both addition and multiplication;
hence ai-semirings are sometimes called \emph{semilattice-ordered semigroups}.
Unless otherwise specified, all references to an order on an ai-semiring are to this one.
We write $a< b$ to indicate that $a\leq b$ and $a\neq b$.

Over the past two decades, the finite basis problem for ai-semirings has been a subject of intensive investigation, attracting considerable attention from researchers worldwide
(see~\cite{d, gv, pas05, shap23} and the references therein).
In particular, the past few years have witnessed significant progress in the study of limit varieties of ai-semirings.

Ren, Jackson, Zhao, and Lei~\cite{rjzl} provided the first explicit examples of limit varieties of ai-semirings.
They presented concrete constructions for an infinite family of such varieties, each generated by a finite flat semiring
(for a more extensive account of flat semirings, see Jackson~\cite{jack08}).
This family arises from a complete characterization of limit varieties generated by flat extensions of finite groups.
In addition, they gave an ad hoc example of a limit variety also generated by a finite flat semiring.
Beyond these, they demonstrated the existence of further examples,
including a continuum-sized family of limit varieties with no finite generator, and two additional ad hoc examples.

More recently, Gao and Ren~\cite{gr} provided another explicit example of a limit variety of ai-semirings,
namely the variety generated by all acyclic graph semirings.

In this paper, we contribute a new example to this growing family of limit varieties of ai-semirings.
The remainder of this paper is organised as follows.
Section~2 collects the necessary preliminaries on commutative ai-semirings,
including the key notions of terms, identities, and flat semirings,
as well as the technical tools that will be used throughout the paper.
In Section~3 we establish our main result, Theorem~\ref{free02},
which provides a sufficient condition for a commutative ai-semiring to be nonfinitely based.
Section~4 is devoted to applications of this criterion.
In Section~5 we describe the subvariety lattice of the variety generated by our main example,
showing that it forms a four-element chain.
We conclude in Section~6 with a brief discussion of some open problems and possible directions for further investigation.

In particular, we show that the six-element ai-semiring $SR_6$,
whose Cayley tables are given in Table~\ref{6-element ai-semirings},
generates a limit variety: $\mathsf{V}(SR_6)$ itself is nonfinitely based,
while every proper subvariety of $\mathsf{V}(SR_6)$ is finitely based. To the best of our knowledge,
$SR_6$ is the smallest known example of an ai-semiring that generates a limit variety,
improving upon the previously known eight-element example. Despite the apparent simplicity of its operations,
in particular the fact that its multiplicative reduct is both commutative and $3$-nilpotent,
$SR_6$ exhibits a surprisingly rich equational structure.

\begin{table}[htbp]
\caption{The Cayley tables of $SR_6$}\label{6-element ai-semirings}
\begin{tabular}{c|cccccc}
$+$               &1 & 2 & 3 & 4 & 5 & 6\\
\hline
                 1& 1 & 1 & 1 & 1 & 1 & 1 \\
                 2&1 & 2 & 1 & 1 & 2 & 1\\
                 3&1 & 1 & 3 & 1 & 1 & 1\\
                 4&1 & 1 & 1 & 4 & 1 & 4\\
                 5&1 & 2 & 1 & 1 & 5 & 1\\
                 6&1 & 1 & 1 & 4 & 1 & 6
\end{tabular}
\qquad\qquad
\begin{tabular}{c|cccccc}
$\cdot$               &1 & 2 & 3 & 4 & 5 & 6\\
\hline
                   1&1 & 1 & 1 & 1 & 1 & 1 \\
                   2&1 & 1 & 1 & 1 & 1 & 3\\
                   3&1 & 1 & 1 & 1 & 1 & 1\\
                   4&1 & 1 & 1 & 1 & 3 & 1\\
                   5&1 & 1 & 1 & 3 & 1 & 3\\
                   6&1 & 3 & 1 & 1 & 3 & 1
\end{tabular}
\end{table}

\section{Preliminaries}
In this section, we introduce the basic terminology, notation, and preliminary results that will be used throughout the paper.
Since our main object of study, $SR_6$, is a commutative ai-semiring,
we focus primarily on notions and facts related to commutative ai-semirings.
Much of the material presented here is drawn from the preliminary section of \cite{yrg}.

Let $X$ be a countably infinite set of variables, and let $X_c^+$ denote the free commutative semigroup over $X$.
A \textit{commutative ai-semiring term} (or simply a \textit{term}) over $X$ is defined as a finite nonempty set of words in $X_c^+$.
(In the sequel, terms are denoted by bold lowercase letters $\mathbf{u}, \mathbf{v}, \mathbf{w}, \dots$,
while ordinary lowercase letters $x, y, z, \dots$ represent variables.)
A term is represented as a formal sum of its elements. Specifically, $\mathbf{w} = \mathbf{u}_1 + \mathbf{u}_2 + \cdots + \mathbf{u}_n$ indicates that $\mathbf{w} = \{ \mathbf{u}_1, \mathbf{u}_2, \dots, \mathbf{u}_n \}$. The order of the summands in the formal sum is irrelevant, and multiple occurrences of the same word are collapsed into a single occurrence. Two terms are equal if and only if their underlying sets coincide.

Let $P_f(X^+_c)$ denote the collection of all terms over $X$.
This set forms a commutative ai-semiring under the usual operations of term addition and multiplication.
It follows from ~\cite[Theorem 2.5]{ku} that $P_f(X^+_c)$ is the free commutative ai-semiring over $X$ in the variety of all commutative ai-semirings.
A \textit{commutative ai-semiring substitution} (or just a \textit{substitution})
is defined as a semiring homomorphism from $P_f(X^+_c)$ to itself.

Let $\mathbf{u}$ and $\mathbf{v}$ be terms.
We say that $\mathbf{u}$ is a \emph{subterm of $\mathbf{v}$} if there exist terms $\mathbf{p}_1$ and $\mathbf{p}_2$ such that
\[
\mathbf{v} = \mathbf{p}_1 \mathbf{u} + \mathbf{p}_2.
\]
Here $\mathbf{p}_1$ may be the empty word (acting as the multiplicative identity),
and $\mathbf{p}_2$ may be the empty set (acting as the additive identity).
The term $\mathbf{v}$ is called $\mathbf{u}$-\emph{free} if for every substitution $\varphi$, the term $\varphi(\mathbf{u})$ is not a subterm of $\mathbf{v}$.

\begin{lem}\label{nlemma0}
Let $\mathbf{u}, \mathbf{v}$ and $\mathbf{w}$ be terms. If $\mathbf{v}$ is $\mathbf{u}$-free and $\mathbf{u}$ is a subterm of $\mathbf{w}$, then $\mathbf{v}$ is also $\mathbf{w}$-free.
\end{lem}
\begin{proof}
This follows directly from the definitions of subterm and freeness.
\end{proof}

A \emph{commutative ai-semiring identity} (or simply an \emph{identity})
is a formal expression of the form $\mathbf{u} \approx \mathbf{v}$, where $\mathbf{u}$ and $\mathbf{v}$ are terms.
Let $S$ be a commutative ai-semiring and $\mathbf{u} \approx \mathbf{v}$ an identity.
We say that $S$ \emph{satisfies $\mathbf{u} \approx \mathbf{v}$}, or that $\mathbf{u} \approx \mathbf{v}$ \emph{holds in $S$}, if $\varphi(\mathbf{u}) = \varphi(\mathbf{v})$ for every semiring homomorphism $\varphi: P_f(X^+_c) \to S$.

The following result, which is a consequence of \cite[Lemma 3.1]{d1}, concerns the equational logic of commutative ai-semirings.

\begin{lem}\label{nlemma1}
Let $\Sigma$ be a set of identities and let $\mathbf{u} \approx \mathbf{v}$ be a nontrivial identity.
Then $\mathbf{u} \approx \mathbf{v}$ is derivable from $\Sigma$ if and only if there exist terms
$\mathbf{t}_1, \mathbf{t}_2, \ldots, \mathbf{t}_n \in P_f(X^+_c)$ such that $\mathbf{u} = \mathbf{t}_1$,
$\mathbf{v} = \mathbf{t}_n$ and, for each $1 \leq i < n$, there are terms
$\mathbf{p}_i, \mathbf{r}_i, \mathbf{s}_i, \mathbf{s}'_i \in P_f(X^+_c)$ and a substitution
$\varphi_i\colon P_f(X^+_c) \to P_f(X^+_c)$ such that
\[
        \mathbf{t}_i = \mathbf{p}_i \varphi_i(\mathbf{s}_i)  + \mathbf{r}_i, \quad
        \mathbf{t}_{i+1} = \mathbf{p}_i \varphi_i(\mathbf{s}'_i) + \mathbf{r}_i,
\]
where $\mathbf{s}_i \approx \mathbf{s}'_i \in \Sigma$ or $\mathbf{s}'_i \approx \mathbf{s}_i \in \Sigma$,
$\mathbf{p}_i$ may be the empty word (acting as the multiplicative identity),
and $\mathbf{r}_i$ may be the empty set (acting as the additive identity).
\end{lem}

The following lemma, which is the compactness theorem of equational logic (see~\cite[Exercise II.14.10]{bs}),
will be our main tool in Section~3 for establishing that a commutative ai-semiring has no finite equational basis.

\begin{lem}\label{nlemma2}
Let \(\mathcal{V}\) be a commutative ai-semiring variety. If \(\mathcal{V}\) is finitely based,
then every equational basis of \(\mathcal{V}\) contains a finite subset that also defines \(\mathcal{V}\).
Equivalently, if \(\mathcal{V}\) has an infinite equational basis none of whose finite subsets defines \(\mathcal{V}\),
then \(\mathcal{V}\) is nonfinitely based.
\end{lem}

We denote by $\mathbf{u} \preceq \mathbf{v}$ (or equivalently $\mathbf{v} \succeq \mathbf{u}$)
the identity $\mathbf{v}\approx \mathbf{v} + \mathbf{u}$, which we refer to as a \textit{commutative ai-semiring inequality} (or simply an \textit{inequality}).
It is routine to verify that a commutative ai-semiring $S$ satisfies the inequality $\mathbf{u} \preceq \mathbf{v}$ if and only if for every semiring homomorphism $\varphi: P_f(X^+_c) \to S$, we have $\varphi(\mathbf{u}) \leq \varphi(\mathbf{v})$.
Consequently, $S$ satisfies an identity $\mathbf{u} \approx \mathbf{v}$ precisely when it satisfies
both inequalities $\mathbf{u} \preceq \mathbf{v}$ and $\mathbf{v} \preceq \mathbf{u}$.
Therefore, $\mathbf{u} \approx \mathbf{v}$ is equivalent to
the inequalities $\mathbf{u} \preceq \mathbf{v}$ and $\mathbf{v} \preceq \mathbf{u}$.
For this reason, when dealing with a set of identities,
we may always assume without loss of generality that it consists of inequalities.
The following discussion shows that it suffices to consider inequalities of a special form.

Now let $\Sigma$ be a set of identities, and let $\mathbf{u} \approx \mathbf{v}$ be an identity such that
\[
\mathbf{u} = \mathbf{u}_1 + \cdots + \mathbf{u}_k, \quad
\mathbf{v} = \mathbf{v}_1 + \cdots + \mathbf{v}_\ell,
\]
where $\mathbf{u}_i, \mathbf{v}_j \in X^+_c$ for $1 \leq i \leq k$ and $1 \leq j \leq \ell$.
One readily checks that the ai-semiring variety defined by $\mathbf{u} \approx \mathbf{v}$ coincides with the ai-semiring variety defined by the inequalities
\[
\mathbf{u}_i \preceq \mathbf{v}, \quad \mathbf{v}_j \preceq \mathbf{u} \quad (1 \leq i \leq k, \, 1 \leq j \leq \ell).
\]
Consequently, to prove that $\mathbf{u} \approx \mathbf{v}$ is derivable from $\Sigma$, it suffices to show that for every $i$ and $j$, the inequalities $\mathbf{u}_i \preceq \mathbf{v}, \, \mathbf{v}_j \preceq \mathbf{u}$ can be derived from $\Sigma$.
In view of this, we always restrict our attention to the inequalities of the form
$\mathbf{q} \preceq \mathbf{u}$, where $\mathbf{q}$ is a word and $\mathbf{u}$ is a term.

To end this section, we introduce an important class of commutative ai-semirings.
For a nonempty subset $W$ of $X_c^+$, let
$S_c(W)$ denote the set consisting of all nonempty subwords of words in $W$ together with a new symbol $0$.
Define operations $+$ and $\cdot$ on $S_c(W)$ by the rule
\[
\bu + \bv =
\begin{cases}
\bu & \text{if } \bu=\bv, \\
0 & \text{otherwise,}
\end{cases}
\qquad
\bu \cdot \bv =
\begin{cases}
\bu\bv & \text{if } \bu\bv \in S_c(W) \setminus \{0\}, \\
0 & \text{otherwise.}
\end{cases}
\]
Then $S_c(W)$ forms a commutative ai-semiring,
where $0$ serves as both the additive top element and the multiplicative zero.
(Such algebras are known as \emph{flat semirings}; see~\cite{jrz, jack08}.)
If $W$ consists of a single word $\bw$, we write $S_c(\bw)$ for $S_c(W)$.

Ren et al.~\cite[Theorem 3.6]{rjzl} showed that the variety $\mathsf{V}(S_c(abc))$ is a limit variety.
Prior to this work, the eight-element algebra $S_c(abc)$ was the smallest known example of an ai-semiring generating a limit variety.

\section{A sufficient condition for the nonfinite basis property}
In this section, we establish a sufficient condition for an ai-semiring to be nonfinitely based.

Let \( n \geq 1 \) be an integer. Define
\[
{\bq}^{(n)} = \prod_{i=1}^{2n+1} x_i,
\]
and
\[
{\bu}^{(n)} = \left(\sum_{i=1}^{2n} x_i x_{i+1}\right)+x_{2n+1} x_{1}.
\]
We denote by $\sigma_n$ the inequality ${\bq}^{(n)} \preceq {\bu}^{(n)}$, and write $\Omega$
for the set of all inequalities $\sigma_n$ with $n \geq 1$.
The inequalities $\sigma_n$ already proved useful in \cite{wrz},
where they were employed to establish the nonfinite basis property of a four-element algebra.

For a word $\bp$, let $\ell(\bp)$ denote the length of $\bp$, that is,
the number of variables occurring in $\bp$ counting multiplicities.
For a term $\bt$, $c(\bt)$ denotes the content of $\bt$, that is,
the set of variables occurring in $\bt$.
For an integer $k\geq 1$,
let $L_k({\bt})$ denote the term that is the sum of all words in $\bt$ of length $k$.

Now define a graph ${\mathbb G}_{\bt}$ with vertex set $\mathsf{V}({\mathbb G}_{\bt})=c(L_2({\bt}))$
and edge set $E({\mathbb G}_{\bt})=\{\{x, y\} \mid xy \in L_2({\bt})\}$.
Note that this graph may contain loops (corresponding to the words of the form $x^2$)
but has no multiple edges.
This graph will play a key role in the analysis that follows.
In particular, for $\mathbf{t} = \mathbf{u}^{(n)}$, the graph $\mathbb{G}_{\mathbf{u}^{(n)}}$ is an odd cycle of length $2n+1$,
with vertices $x_1,\dots,x_{2n+1}$ and edges $\{x_i, x_{i+1}\}$ (indices taken modulo $2n+1$).

Define the odd path closure ${\mathbb G}_{\bt}^{odd}$ of ${\mathbb G}_{\bt}$ by
\[
{\mathbb G}_{\bt}^{odd}=\{xy\mid\text{there is an odd path in ${\mathbb G}_{\bt}$ between $x$ and $y$}\}.
\]
(This notation will be used later in Lemma~\ref{lem01}.)

\begin{pro}\label{profree}
Let $n \geq 1$ be an integer. Then ${\bu}^{(m)}$ is ${\bu}^{(n)}$-free for every $m>n$.
\end{pro}
\begin{proof}
Seeking a contradiction, assume that for some $1 \leq n < m$,
${\bu}^{(m)}$ is not ${\bu}^{(n)}$-free.
Then there exist terms ${\bp}, {\br} \in P_f(X_c^+)$ and a substitution $\varphi$ such that
\[
{\bp}\varphi({\bu}^{(n)}) + {\br} = {\bu}^{(m)},
\]
where ${\bp}$ may be the empty word and ${\br}$ may be the empty set.

Observe that both ${\bu}^{(n)}$ and ${\bu}^{(m)}$ consist solely of words of length $2$,
and that any three distinct words in ${\bu}^{(m)}$ have no variable in common.
It follows that ${\bp}$ must be the empty word, and for every variable $x$ occurring in ${\bu}^{(n)}$,
the image $\varphi(x)$ is either a single variable or a sum of two distinct variables.
Hence we obtain
\begin{equation}\label{un01}
\varphi({\bu}^{(n)}) + \mathbf{r} = {\bu}^{(m)}.
\end{equation}

Without loss of generality, suppose that $\varphi(x_1)$ is a sum of two distinct variables.
Then $\varphi(x_1) = x_{j-1} + x_{j+1}$ for some $j$, where the indices are taken modulo $2m+1$.
All subsequent index arithmetic is understood modulo $2m+1$.
From
\[
(x_{j-1}+x_{j+1})\varphi(x_2)=\varphi(x_1)\varphi(x_2)=\varphi(x_1x_2) \subseteq {\bu}^{(m)},
\]
we must have that $\varphi(x_2)=x_j$.
Similarly, from
\[
\varphi(x_{2n+1})(x_{j-1}+x_{j+1}) =\varphi(x_{2n+1})\varphi(x_1)=\varphi(x_{2n+1}x_1) \subseteq \mathbf{u}^{(m)},
\]
we obtain that $\varphi(x_{2n+1}) = x_j$ as well.

Now consider $\varphi(x_2x_3) = \varphi(x_2)\varphi(x_3)=x_j \varphi(x_3) \subseteq {\bu}^{(m)}$.
Since $x_j$ already appears in the two distinct words $x_{j-1}x_j$ and $x_jx_{j+1}$ from the previous steps,
to avoid $x_j$ appearing in a third distinct word,
$c(\varphi(x_3))$ must be a subset of $\{x_{j-1}, x_{j+1}\}$.
Similarly, from $\varphi(x_3x_4)\subseteq {\bu}^{(m)}$ we deduce that $c(\varphi(x_4))$ is a subset of $\{x_{j-2}, x_j, x_{j+2}\}$.

Continuing this process inductively, we obtain that for each $k \geq 1$:
\[
c(\varphi(x_{2k+1})) \subseteq \{x_{j-2k+1}, x_{j-2k+3}, \dots, x_{j+2k-3}, x_{j+2k-1}\},
\]
\[
c(\varphi(x_{2k+2})) \subseteq \{x_{j-2k}, x_{j-2k+2}, \dots, x_{j+2k-2}, x_{j+2k}\}.
\]
Therefore, $c(\varphi(x_{2n+1}))$ is contained in the set
\[
\{x_{j-2n+1}, x_{j-2n+3}, \dots, x_{j+2n-3}, x_{j+2n-1}\},
\]
which does not contain $x_j$.
This contradicts the previously established fact that $\varphi(x_{2n+1})=x_j$.

Therefore, our initial assumption that $\varphi(x_1)$ is a sum of two distinct variables is false.
Hence $\varphi(x)$ is a single variable for every $x \in c({\bu}^{(n)})$.
Let $y_i$ denote the variable $\varphi(x_i)$, $1\leq i \leq 2n+1$.
From \eqref{un01} we have
\[
y_1y_2,\; y_2y_3,\; \dots,\; y_{2n}y_{2n+1},\; y_{2n+1}y_1 \subseteq \mathbf{u}^{(m)}.
\]
Thus the sequence $y_1, y_2, \dots, y_{2n+1}, y_1$ forms a closed walk of length $2n+1$ in the graph ${\mathbb G}_{\mathbf{u}^{(m)}}$.

A basic fact from graph theory states that every closed walk of odd length contains an odd cycle.
Since $\mathbb{G}_{\mathbf{u}^{(m)}}$ itself is an odd cycle of length $2m+1$,
the only odd cycle it can contain is the entire cycle.
Consequently, we must have that $2n+1 = 2m+1$, that is, $n = m$, contradicting $n < m$.
Therefore, $\mathbf{u}^{(m)}$ is $\mathbf{u}^{(n)}$-free.
\end{proof}

\begin{thm}\label{free02}
Let $S$ be a commutative ai-semiring and $\Sigma$ an equational basis for $S$ that contains an infinite subset of $\Omega$.
If $\mathbf{u}^{(m)}$ is $\mathbf{t}$-free for every inequality $\mathbf{s} \preceq \mathbf{t}$ in $\Sigma \setminus \Omega$ and for every integer $m \geq 1$, then $S$ is nonfinitely based.
\end{thm}

\begin{proof}
By the compactness theorem of equational logic (see Lemma~\ref{nlemma2}),
it is enough to prove that no finite subset of the set $\Sigma$ defines the variety $\mathsf{V}(S)$.

Let $\Sigma'$ be an arbitrary finite subset of $\Sigma$. Since $\Sigma$ contains an infinite subset of $\Omega$,
we can choose an integer $m$ such that
\[
m > \max\{n \geq 1 \mid \sigma_n \in \Sigma'\} \quad \text{and} \quad \sigma_m \in \Sigma.
\]
Then $\sigma_m \notin \Sigma'$.

To prove that $\Sigma'$ cannot define the variety $\mathsf{V}(S)$,
it suffices to verify that $\sigma_m$ is not derivable from $\Sigma'$.
By hypothesis, together with Lemmas \ref{nlemma0} and \ref{nlemma1},
it is sufficient to show that $\mathbf{u}^{(m)}$ is $\mathbf{u}^{(n)}$-free for all $1 \leq n < m$,
which follows immediately from Proposition~\ref{profree}. This completes the proof.
\end{proof}

\section{Applications of Theorem~\ref{free02}}
This section is devoted to applications of Theorem~\ref{free02}.
We prove that three ai-semirings, in particular $SR_6$, are nonfinitely based.

To this end, we first need an infinite equational basis for $SR_6$.
This requires a sufficiently detailed understanding of the identities satisfied by $SR_6$.
We begin by presenting some identities that hold in $SR_6$.

\begin{lem}\label{lem26022201}
$SR_6$ satisfies the identities $\sigma_n$ for all $n \geq 1$, together with
\begin{align}
&x^3 \approx x^2; \label{SR02}\\
&x^2 \approx x+xy; \label{SR03}\\
&x_1 \preceq x_2x_3x_4.\label{SR04}
\end{align}
\end{lem}
\begin{proof}
It is straightforward to verify that $SR_6$ satisfies \eqref{SR02}, \eqref{SR03} and \eqref{SR04},
while the main task is to prove that $\sigma_n$ holds in $SR_6$ for all $n \geq 1$.
Indeed, let $\varphi\colon P_f(X_c^+) \to SR_6$ be an arbitrary semiring homomorphism.
Observe that $1$ is the additive top element of $SR_6$.
By \eqref{SR04} we obtain that $\varphi(\bq^{(n)})=1$.
It suffices to show that $\varphi(\bu^{(n)})=1$.
From the multiplicative Cayley table of $SR_6$ one sees that the product of any two elements is either $1$ or $3$;
consequently, $\varphi(\bw)=1$ or $3$ for every $\bw \in \bu^{(n)}$.
We only need to rule out the possibility that $\varphi(\bw)=3$ for every $\bw \in \bu^{(n)}$.

Suppose, to the contrary, that $\varphi(x_i x_{i+1}) = 3$ for all $i = 1, 2, \dots, 2n+1$, where we set $x_{2n+2} := x_1$ for convenience.
Then $\varphi(x_i)\varphi(x_{i+1}) = 3$ for each $i$.
From the multiplicative Cayley table of $SR_6$, the only unordered pairs $\{a,b\}$ with $ab = 3$ are
\[
\{2,6\},\;\{4,5\},\;\{5,6\}.
\]
Thus each $\varphi(x_i)$ must belong to $\{2,4,5,6\}$, and consecutive elements must form one of these three unordered pairs.
Define a graph $\mathbb{G}$ with vertex set $\{2,4,5,6\}$ and edge set $\{\{2,6\}, \{4,5\}, \{5,6\}\}$.
Observe that $\mathbb{G}$ is simply a path of length $3$, so it contains no cycles at all.
Now consider the closed walk
\[
\varphi(x_1), \varphi(x_2), \dots, \varphi(x_{2n+1}), \varphi(x_1),
\]
which has length $2n+1$.
Since $\mathbb{G}$ has no cycles, such a closed walk of odd length cannot exist in $\mathbb{G}$.
Hence $\varphi(x_i x_{i+1}) = 1$ for some $i$, and so $\varphi(\mathbf{u}^{(n)}) = 1$.
Therefore, $\sigma_n$ holds in $SR_6$.
\end{proof}

Observe that $S_c(ab)$ embeds into $SR_6$;
more precisely, $S_c(ab)$ is isomorphic to the subalgebra $\{1,3,5,6\}$ of $SR_6$.
By Lemma~\ref{lem26022201}, $S_c(ab)$ satisfies the identities \eqref{SR02}, \eqref{SR03}, \eqref{SR04},
and $\sigma_n$ for all $n \geq 1$.
Moreover, it is easy to check that the inequality
\begin{equation}\label{26022301}
x_1x_4\preceq x_1x_2+x_2x_3+x_3x_4
\end{equation}
holds in $S_c(ab)$.
From this, one can derive the more general family of inequalities
\[
x_1x_{2n+2}\preceq x_1x_2+x_2x_3+\cdots+x_{2n+1}x_{2n+2}\quad (n\geq 1),
\]
which we denote by $\delta_n$ ($n \geq 1$). These identities are also satisfied by $S_c(ab)$.
We now proceed to give a complete description of the identities of $S_c(ab)$.

\begin{lem}\label{lem01}
Let $\bq\preceq \bu$ be a nontrivial inequality such that
$\bu=\bu_1+\bu_2+\cdots+\bu_n$ with $\bu_i, \bq \in X^+_c$ for $1\leq i \leq n$.
Then $\bq\preceq \bu$ is satisfied by $S_c(ab)$ if and only if $\bu$ and $\bq$ satisfy one of the following conditions\up:
\begin{enumerate}[$(\rm i)$]
\item  $\ell(\bu_i)\geq 3$ for some $\bu_i\in \bu$;

\item  $c(L_1(\bu))\cap c(L_2(\bu))\neq \emptyset$;

\item  the graph ${\mathbb G}_{\bu}$ contains an odd cycle;

\item  $\bq\in {\mathbb G}_{\bu}^{odd}$.
\end{enumerate}
\end{lem}
\begin{proof}
We first prove sufficiency.
Suppose that one of the conditions (i)--(iii) holds.
Let $\varphi: P_f(X^+_c)\to S_c(ab)$ be an arbitrary semiring homomorphism.
We shall show that $\varphi(\mathbf{u})=0$, and then
\[
\varphi(\mathbf{u})=0=0+\varphi(\mathbf{q})=\varphi(\mathbf{u})+\varphi(\mathbf{q})=\varphi(\mathbf{u}+\mathbf{q}),
\]
since $0$ is the additive top element of $S_c(ab)$.

If the condition (i) holds, or if ${\mathbb G}_{\bu}$ contains a loop,
then $\varphi(\mathbf{u})=0$, since $S_c(ab)$ satisfies the identities \eqref{SR02} and \eqref{SR04}.
If the condition (ii) holds, or if ${\mathbb G}_{\bu}$ contains an odd cycle with length at least $3$, then $\varphi(\mathbf{u})=0$,
since $S_c(ab)$ satisfies \eqref{SR02}, \eqref{SR03}, \eqref{SR04}, and $\delta_n$, $n\geq 1$.

Now suppose that the condition (iv) holds. Then there exist $x_1$, $x_2, \ldots$, $x_{2n+1}$, $x_{2n+2}\in c(L_2(\bu))$
for some integer $n\geq 1$ such that $\bq=x_1x_{2n+2}$,
and
\[
x_1x_2,\;  x_2x_3, \; \ldots,\;  x_{2n}x_{2n+1}, \; x_{2n+1}x_{2n+2}\in L_2(\bu).
\]
Applying the inequality $\delta_n$, we obtain
\[
\bu\succeq x_1x_2+x_2x_3\cdots +x_{2n}x_{2n+1}+x_{2n+1}x_{2n+2}\succeq x_1x_{2n+2}=\bq.
\]
Thus $\bq\preceq \bu$ holds in $S_c(ab)$.

We now prove necessity.
Suppose that $S_c(ab)$ satisfies $\mathbf{q}\preceq \mathbf{u}$. We shall show that one of conditions (i)--(iv) must hold.

Assume that (i)--(iii) are false. We will show that (iv) holds.
Under this assumption, $\mathbf{u} = L_1(\mathbf{u}) + L_2(\mathbf{u})$ with
$c(L_1(\mathbf{u})) \cap c(L_2(\mathbf{u})) = \emptyset$, all words in $L_2(\mathbf{u})$ are linear,
and the graph $\mathbb{G}_{\mathbf{u}}$ contains no odd cycles.
Hence $\mathbb{G}_{\mathbf{u}}$ is bipartite, and so we can write
\[
L_2(\mathbf{u}) = x_1y_1 + x_2y_2 + \cdots + x_ny_n,
\]
where $\{x_1, x_2, \dots, x_n\}$ and $\{y_1, y_2, \dots,y_n\}$ are disjoint.
Set $A = \{x_1, x_2, \dots, x_n\}$, $B = \{y_1, y_2, \dots, y_n\}$.

Define a semiring homomorphism $\varphi\colon P_f(X^+_c) \to S_c(ab)$ by: for any $x \in X$,
\[
\varphi(x) =
\begin{cases}
ab, & x \in c(L_1(\mathbf{u})), \\[2pt]
a,  & x \in A, \\[2pt]
b,  & x \in B, \\[2pt]
0,  & \text{otherwise}.
\end{cases}
\]
Then
\[
\varphi(\mathbf{u})=\varphi(L_1(\mathbf{u}) + L_2(\mathbf{u}))=\varphi(L_1(\mathbf{u}))+\varphi(L_2(\mathbf{u}))=ab+ab=ab.
\]
Since $\mathbf{q}\preceq \mathbf{u}$ holds in $S_c(ab)$,
we must have $\varphi(\mathbf{q}) \leq \varphi(\mathbf{u})$ and consequently $\varphi(\mathbf{q}) = ab$.

Because $S_c(ab)$ satisfies \eqref{SR02} and \eqref{SR04},
$\mathbf{q}$ must be a linear word of length at most $2$.
If $\ell(\mathbf{q}) = 1$, then $\mathbf{q} \in L_1(\mathbf{u})$,
which would make $\mathbf{q}\preceq \mathbf{u}$ trivial, a contradiction.
Hence $\ell(\mathbf{q}) = 2$; write $\mathbf{q} = xy$ with $|c(\mathbf{q}) \cap A| = |c(\mathbf{q}) \cap B| = 1$.
Without loss of generality, assume that $x \in A$ and $y \in B$.

Now define the neighborhood $N(x)$ of $x$ in $\mathbb{G}_{\mathbf{u}}$:
\[
N(x) = \{ z \in V(\mathbb{G}_{\mathbf{u}}) \mid z = x \text{ or there is a path between } z \text{ and } x \}.
\]
Consider another semiring homomorphism $\psi: P_f(X^+_c) \to S_c(ab)$ by: for any $t \in X$,
\[
\psi(t) =
\begin{cases}
ab, & t \in c(L_1(\mathbf{u})), \\[2pt]
a,  & t \in (A \setminus N(x)) \cup (B \cap N(x)), \\[2pt]
b,  & t \in (B \setminus N(x)) \cup (A \cap N(x)), \\[2pt]
0,  & \text{otherwise}.
\end{cases}
\]
Then $\psi(\mathbf{u})=ab$ and $\psi(x)=b$.
This forces $\psi(y) = a$, which implies $y \in B \cap N(x)$.
Since $\mathbb{G}_{\mathbf{u}}$ is bipartite, it follows that
there exists a path of odd length between $x$ and $y$.
Therefore, $\bq\in {\mathbb G}_{\bu}^{odd}$.
\end{proof}

\begin{pro}\label{pro26012310}
$\mathsf{V}(S_c(ab))$ is the commutative ai-semiring variety defined by the identities
\eqref{SR02}, \eqref{SR03}, \eqref{SR04}, and \eqref{26022301}.
\end{pro}
\begin{proof}
We know that $\mathsf{V}(S_c(ab))$ satisfies the identities
\eqref{SR02}, \eqref{SR03}, \eqref{SR04}, and \eqref{26022301}.
It remains to show that every inequality of $S_c(ab)$ is derivable from
\eqref{SR02}, \eqref{SR03}, \eqref{SR04}, and \eqref{26022301}.
Let $\bq \preceq \bu$ be such a nontrivial inequality,
where
$\bu=\bu_1+\bu_2+\cdots+\bu_n$ with $\bu_i, \bq \in X^+_c$ for $1\leq i \leq n$.
By Lemma \ref{lem01} it is enough to consider the following four cases.

\textbf{Case 1.} $\ell(\bu_i)\geq 3$ for some $\bu_i\in \bu$. Then
\[
\bu \succeq \bu_i \stackrel{\eqref{SR04}}\succeq \bq.
\]
This implies $\bu \succeq \bq$.

\textbf{Case 2.} $c(L_1(\bu))\cap c(L_2(\bu))\neq\emptyset$.
Take $x$ in $c(L_1(\bu))\cap c(L_2(\bu))$.
Then $xy\in L_2(\bu)$ for some $y\in X$.
Now we have
\[
\bu \succeq x+xy \stackrel{\eqref{SR03}}\approx x^2 \stackrel{\eqref{SR02}}\approx x^3 \stackrel{\eqref{SR04}}\succeq \bq.
\]
This derives $\bu \succeq \bq$.

\textbf{Case 3.}  The graph ${\mathbb G}_{\bu}$ contains an odd cycle.
Then $L_2(\bu)$ contains
$x_1x_2, x_2x_3,$ $ x_3x_4, \ldots, x_{2k+1}x_1$
for some $x_1, x_2, \ldots, x_{2k+1}\in X$ and $k\geq1$.
Then
\[
\bu \succeq x_1x_2+x_2x_3+\cdots+x_{2k+1}x_1 \stackrel{\delta_k}\succeq x_1^2 \stackrel{\eqref{SR02}}\approx x_1^3 \stackrel{\eqref{SR04}}\succeq \bq.
\]
This implies $\bu \succeq \bq$.

\textbf{Case 4.} $\bq\in {\mathbb G}_{\bu}^{odd}$.
Then $\bq=xy$ for some $x, y\in X$, and there exists an odd path in ${\mathbb G}_{\bu}$ between $x$ and $y$,
whose length is at least three.
So we may assume that $x_1x_2, x_2x_3, \ldots, x_{2n+1}x_{2n+2}\in L_2(\bu)$ for some $n\geq 1$,
where $x_1=x$ and $x_{2n+2}=y$.
Now
\[
\bu\succeq x_1x_2+x_2x_3+\cdots+x_{2n+1}x_{2n+2}\stackrel{\delta_n}\succeq x_1x_{2n+2}=xy=\bq.
\]
This derives $\bu \succeq \bq$.
\end{proof}

\begin{remark}\label{rem:basis-comparison}
It follows from \cite[Proposition 3.2]{rjzl} and \cite[Lemma 2.1]{jrz} that $\mathsf{V}(S_c(ab))$ admits another finite equational basis. This basis, however, contains more identities than the one given in Proposition~\ref{pro26012310} and is obtained via a completely different approach, namely by analyzing the subdirectly irreducible members of the variety. More importantly, it is less convenient for the purposes of our subsequent discussion on the subvariety lattice of $\mathsf{V}(SR_6)$.
\end{remark}

We now present a countably infinite equational basis for $SR_6$.

\begin{pro}\label{proSR}
$\mathsf{V}(SR_6)$ is the commutative ai-semiring variety defined by the identities
\eqref{SR02}, \eqref{SR03}, \eqref{SR04} and $\sigma_n$ for all $n \geq 1$.
\end{pro}
\begin{proof}
By Lemma~\ref{lem26022201}, $SR_6$ satisfies the identities \eqref{SR02}, \eqref{SR03}, \eqref{SR04} and $\sigma_n$ for all $n \geq 1$.
Thus it suffices to show that every inequality holding in $SR_6$ is derivable from these identities.
Let $\mathbf{q} \preceq \mathbf{u}$ be such a nontrivial inequality,
where
$\bu=\bu_1+\bu_2+\cdots+\bu_n$ with $\bu_i, \bq \in X^+_c$ for $1\leq i \leq n$.
Since $S_c(ab)$ embeds into $SR_6$, the same inequality also holds in $S_c(ab)$.
By Lemma~\ref{lem01}, it is enough to consider the following four cases.

The proofs for Cases 1 and 2 are identical to those given in the proof of Proposition~\ref{pro26012310} for $S_c(ab)$; we omit them here.

\textbf{Case 3.} The graph $\mathbb{G}_{\mathbf{u}}$ contains an odd cycle.
Then $L_2(\mathbf{u})$ contains words of the form
\[
x_1x_2,\; x_2x_3,\; x_3x_4,\; \ldots,\; x_{2k+1}x_1
\]
for some $x_1, x_2, \ldots, x_{2k+1} \in X$ and $k \geq 1$.
Consequently,
\[
\mathbf{u} \succeq x_1x_2 + x_2x_3 + \cdots + x_{2k+1}x_1
        \stackrel{\sigma_k}{\succeq} x_1x_2x_3\cdots x_{2k+1}
        \stackrel{\eqref{SR04}}{\succeq} \mathbf{q}.
\]
This implies $\bu \succeq \bq$.

\textbf{Case 4.} $\bq\in {\mathbb G}_{\bu}^{odd}$.
Then $\bq=xy$ for some $x, y\in X$, and there exists is an odd path in ${\mathbb G}_{\bu}$ between $x$ and $y$,
whose length is at least three. Since Cases 1--3 are already excluded,
we may also assume that $\ell(\bu_k) \leq 2$ for all $\bu_k\in \bu$, $\bu_k$ is linear for all $\bu_k \in L_2(\bu)$,
$c(L_1(\bu))\cap c(L_2(\bu))=\emptyset$, and ${\mathbb G}_{\bu}$ contains no odd cycles.
Thus the graph ${\mathbb G}_{\bu}$ is bipartite.

Consequently, we may write
\[
L_2(\bu) = x_1y_1 + x_2y_2 +\cdots + x_ny_n,
\]
where $\{x_1, x_2, \ldots, x_n\}$ and $\{y_1, y_2, \ldots, y_n\}$ are disjoint.
Set $A = \{x_1, x_2, \ldots, x_n\}$ and $B = \{y_1, y_2, \ldots, y_n\}$.
Let $\alpha\colon P_f(X^+_c) \to SR_6$ be a semiring homomorphism such that
$\alpha(z)=5$ if $z \in A$, $\alpha(z)=6$ if $z \in B$, and $\alpha(z)=3$ otherwise.
Then $\alpha(\bu)=3$ and since $\alpha(\mathbf{q}) \leq \alpha(\mathbf{u})$,
we must have that $\alpha(\mathbf{q}) = 3$.
This implies that $x \neq y$; without loss of generality, suppose $x \in A$ and $y \in B$.

Now consider the semiring homomorphism $\varphi\colon P_f(X^+_c) \to SR_6$ by: for any $z \in X$,
\[
\varphi(z) =
\begin{cases}
2,  & z=x, \\[2pt]
4,  & z=y, \\[2pt]
5,  & z \in A\backslash\{x\}, \\[2pt]
6,  & z \in B\backslash\{y\}, \\[2pt]
3,  & \text{otherwise}.
\end{cases}
\]
Then
\[
\varphi(\bq)=\varphi(xy)=\varphi(x)\varphi(y)=2\cdot4=1
\]
and
\[
\varphi(\bu)=\varphi(L_1(\bu)+L_2(\bu))=\varphi(L_1(\bu))+\varphi(L_2(\bu))=3+2\cdot 6+5\cdot6+5\cdot 4=3.
\]
Thus $1\leq 3$, a contradiction.
Hence Case 4 cannot occur.
\end{proof}

\begin{cor}\label{cor26022310}
$\mathsf{V}(S_c(ab))$ is the subvariety of $\mathsf{V}(SR_6)$ defined by the inequality \eqref{26022301}.
\end{cor}
\begin{proof}
This follows from Propositions \ref{pro26012310} and \ref{proSR} immediately.
\end{proof}

\begin{cor}\label{LTN}
The ai-semiring $SR_6$ is nonfinitely based.
\end{cor}

\begin{proof}
By Proposition \ref{proSR},
\[
\Sigma=\{\eqref{SR02}, \eqref{SR03}, \eqref{SR04}\} \cup \{\sigma_{n} \mid n \geq 1\}
\]
is an equational basis for the  ai-semiring $SR_6$, and it contains \( \Omega\).
By Theorem \ref{free02}, to prove that $SR_6$ is nonfinitely based,
it suffices to show that for any $m \geq 1$,
${\bu}^{(m)}$ is ${\bw}$-free for each ${\bw}$ in the set $\{x^3, x^2, x + xy, xyz\}$.

The required freeness follows directly from the basis properties of ${\bu}^{(m)}$,
which can be obtained from its explicit form:
${\bu}^{(m)}$ has no subterms of the form $\bt^2$, $\bt_1+\bt_1\bt_2$, or $\bt_1\bt_2\bt_3$.
\end{proof}

Let $S$ be an ai-semiring. One can construct an ai-semiring $S^0$ by adjoining a new element $0$ to $S$.
The operations on $S^0 = S \cup \{0\}$ are defined by extending those on $S$ as follows:
for every $a \in S \cup \{0\}$,
\[
a + 0 = 0 + a = a,\quad a \cdot 0 = 0 \cdot a = 0,
\]
while the original operations on $S$ remain unchanged.
Thus $S$ is a subsemiring of $S^0$, and $0$ serves as both the additive least element and the multiplicative zero.
In particular, if $S$ is the trivial ai-semiring, then $S^0$ is a two-element distributive lattice; this algebra is denoted by $D_2$.

Next, we prove that the ai-semiring variety $\mathsf{V}(S_c(ab), D_2)$ generated by $S_c(ab)$ and $D_2$ is nonfinitely based.
To this end, we first recall a description of the identities of $D_2$, which follows directly from \cite[Lemma 1.1]{sr}.

\begin{lem}\label{lemd2}
Let $\bq\preceq\bu$ be a nontrivial inequality such that
$\bu=\bu_1+\bu_2+\cdots+\bu_n$ and $\bu_i, \bq \in X^+$ for all $1\leq i \leq n$. Then
$\bq\preceq\bu$ holds in $D_2$ if and only if $c(\bu_i) \subseteq c(\bq)$ for some $\bu_i \in \bu$.
\end{lem}

\begin{pro}\label{proScabD2}
$\mathsf{V}(S_c(ab), D_2)$ is the commutative ai-semiring variety defined by the identities
\eqref{SR02}, \eqref{SR03}, $\sigma_n$ for all $n \geq 1$, and
\begin{align}
&x_1x_5 \preceq x_1+x_2x_3x_4.\label{ScabD201}
\end{align}
\end{pro}
\begin{proof}
We know that $S_c(ab)$ satisfies the identities \eqref{SR02}, \eqref{SR03}, and $\sigma_n$ for all $n \geq 1$.
It is straightforward to verify that $S_c(ab)$ also satisfies inequality~\eqref{ScabD201}.
By Lemma~\ref{lemd2}, the identities \eqref{SR02}, \eqref{SR03}, \eqref{ScabD201}, and $\sigma_n$ ($n \geq 1$) also hold in $D_2$.

It remains to show that every inequality that holds in both $S_c(ab)$ and $D_{2}$
is derivable from \eqref{SR02}, \eqref{SR03}, \eqref{ScabD201}, and $\sigma_n$ for all $n \geq 1$.
Let $\bq\preceq \bu$ be such an inequality, where
$\bu=\bu_1+\bu_2+\cdots+\bu_n$ and $\bu_i, \bq \in X_c^+$, $1 \leq i \leq n$.
Lemma \ref{lemd2} implies that there exists $\bu_i\in\bu$ such that $c(\bu_i)\subseteq c(\bq)$.
By Lemma \ref{lem01}, it suffices to consider the following four cases.

\textbf{Case 1.} $\ell(\bu_j)\geq 3$ for some $\bu_j\in \bu$. Then
\[
\bu \succeq \bu_j+\bu_i \stackrel{\eqref{ScabD201}}\succeq \bu_i\bq \stackrel{\eqref{SR02}, \eqref{SR03}}\succeq \bq,
\]
which yields $\bu \succeq \bq$.

\textbf{Case 2.} $c(L_1(\bu))\cap c(L_2(\bu))\neq\emptyset$.
The proof is the same as that of Case 2 in Proposition~\ref{pro26012310}.

\textbf{Case 3.} The graph $\mathbb{G}_{\bu}$ contains an odd cycle.
Following the argument used in Proposition~\ref{proSR},
one obtains the inequality $\mathbf{u} \succeq x_1x_2 \cdots x_{2k+1}$.
The rest of the proof proceeds as in Case 1.

\textbf{Case 4.} $\bq\in {\mathbb G}_{\bu}^{odd}$.
Then $\bq=xy$ for some $x, y\in X$, and there exists an odd path in ${\mathbb G}_{\bu}$ between $x$ and $y$.
This implies that $x, y\in c(L_2(\bu))$, and so $c(\bu_i)\subseteq c(L_2(\bu))$.
Since Cases 1--3 are already excluded,
we may also assume that $\ell(\bu_k) \leq 2$ for all $\bu_k\in \bu$, $\bu_k$ is linear for all $\bu_k \in L_2(\bu)$,
and $c(L_1(\bu))\cap c(L_2(\bu))=\emptyset$.
Then $\bu_i\in L_2(\bu)$ and so $\bu_i = \bq$.
Thus $\bq\preceq \bu$ is trivial.
\end{proof}

\begin{cor}
The ai-semiring variety $\mathsf{V}(S_c(ab), D_2)$ is nonfinitely based.
\end{cor}
\begin{proof}
The proof is similar to that of Corollary~\ref{LTN} and is therefore omitted.
\end{proof}

Finally, we prove that the ai-semiring $S_c(ab)^0$ is nonfinitely based.
For this purpose, we need the following lemma due to Wu et al.~\cite[Proposition 1.5]{wrz},
which relates the equational theories of $S^0$ and $S$.

\begin{lem}\label{lem000}
Let $\bq\preceq \bu$ be an ai-semiring identity such that
$\bu=\bu_1+\bu_2+\cdots+\bu_n$, where $\bu_i, \bq \in X^+_c$, $1\leq i \leq n$.
Then $\bq\preceq \bu$ is satisfied by ${S}^0$ if and only if
$\bq\preceq D_\bq(\bu)$ holds in $S$,
where $D_\bq(\bu)$ denotes the term that is the sum of words $\bu_i$ in $\bu$ with $c(\bu_i)\subseteq c(\bq)$.
\end{lem}

\begin{pro}\label{proab0}
$\mathsf{V}(S_c(ab)^0)$
is the commutative ai-semiring variety defined by the identities \eqref{SR02}, $\sigma_n$ for all $n \geq 1$, and
\begin{align}
&x \preceq x^2;\label{ab01}\\
&xyzt \preceq xyz;\label{ab02}\\
&(xy)^2 \preceq x+xy.\label{ab03}
\end{align}
\end{pro}
\begin{proof}
It is straightforward to verify that $S_c(ab)^0$ satisfies the identities \eqref{SR02}, \eqref{ab01}--\eqref{ab03}.
By Lemmas \ref{lem01} and \ref{lem000}, $S_c(ab)^0$ also satisfies the inequalities in $\sigma_n$ for all $n \geq 1$.
It suffices to show that every inequality holding in $S_c(ab)^0$
is derivable from \eqref{SR02}, \eqref{ab01}--\eqref{ab03}, and $\sigma_n$ for all $n \geq 1$.
Consider such an inequality $\bq\preceq \bu$, where
$\bu=\bu_1+\bu_2+\cdots+\bu_n$ and $\bu_i, \bq \in X_c^+$, $1 \leq i \leq n$.
By Lemma~\ref{lem000}, $\bq\preceq D_\bq(\bu)$ holds in $S_c(ab)$.
Combined with Lemma \ref{lem01}, it suffices to consider the following four cases.

\textbf{Case 1.} $\ell(\bu_i)\geq 3$ for some $\bu_i\in D_\bq(\bu)$. Then
\[
\bu \succeq \bu_i \stackrel{\eqref{ab02}}\succeq \bu_i\bq \stackrel{\eqref{SR02}, \eqref{ab01}}\succeq \bq,
\]
which implies $\bu \succeq \bq$.

\textbf{Case 2.} $c(L_1(D_\bq(\bu)))\cap c(L_2(D_\bq(\bu)))\neq\emptyset$.
Take a variable $x$ in this set.
Then $xy\in L_2(D_\bq(\bu))$ for some $y\in X$.
Now we have
\[
\bu \succeq x+xy \stackrel{\eqref{ab03}}\succeq (xy)^2,
\]
and the rest of the proof proceeds as in Case 1.

\textbf{Case 3.} The graph $\mathbb{G}_{D_\bq(\bu)}$ contains an odd cycle.
Following the argument used in the proof of Proposition~\ref{proSR},
one obtains the inequality $\mathbf{u} \succeq x_1x_2\cdots x_{2k+1}$.
The remainder of the proof is analogous to Case 1.

\textbf{Case 4.} $\bq\in {\mathbb G}_{D_\bq(\bu)}^{odd}$.
The proof is the same as that of Case 4 in Proposition~\ref{proScabD2}.
\end{proof}

\begin{cor}\label{cor26030401}
The ai-semiring $S_c(ab)^0$ is nonfinitely based.
\end{cor}
\begin{proof}
This follows by an argument analogous to that of Corollary~\ref{LTN}; thus we omit the details.
\end{proof}

\begin{remark}
The nonfinite basis property of $S_c(ab)^0$ has also been established by Zhao and Wu~\cite{zw},
using a completely different method.
Moreover, by Propositions~\ref{proSR}, \ref{proScabD2} and \ref{proab0}, we have the following strict inclusions:
\[
\mathsf{V}(SR_6) \subset \mathsf{V}(S_c(ab), D_2) \subset \mathsf{V}(S_c(ab)^0).
\]
\end{remark}

\section{The subvariety lattice of the variety $\mathsf{V}(SR_6)$}
In this section, we describe the structure of the subvariety lattice of the variety $\mathsf{V}(SR_6)$.
As a consequence, we show that $\mathsf{V}(SR_6)$ is a limit variety.
The first step is to determine the minimal nontrivial subvarieties of $\mathsf{V}(SR_6)$.

Observe that the two-element commutative ai-semiring $S_c(a)$,
which is denoted by $T_2$ in \cite{sr}, is a subalgebra of $S_c(ab)$.
Hence $\mathsf{V}(S_c(a))$ is a subvariety of $\mathsf{V}(SR_6)$.
In fact, we have the following result.

\begin{pro}\label{T}
 $\mathsf{V}(S_c(a))$ is the only minimal nontrivial subvariety of $\mathsf{V}(SR_6)$.
\end{pro}
\begin{proof}
From \cite[Theorem 1.1]{p} and \cite[Table 1]{sr}, it is known that
$\mathsf{V}(L_2)$, $\mathsf{V}(R_2)$, $\mathsf{V}(M_2)$, $\mathsf{V}(D_2)$,
$\mathsf{V}(N_2)$, and $\mathsf{V}(S_c(a))$ form a complete list of the minimal nontrivial subvarieties of the
variety of all ai-semirings. We know that $S_c(a)$ lies in $\mathsf{V}(SR_6)$,
whereas none of $L_2$, $R_2$, $M_2$, $D_2$ and $N_2$ satisfy \eqref{SR04}.
Consequently, $\mathsf{V}(S_c(a))$ is the only minimal nontrivial subvariety of $\mathsf{V}(SR_6)$.
\end{proof}

\begin{pro}\label{prot1}
$\mathsf{V}(S_c(a))$ is  the subvariety of $\mathsf{V}(SR_6)$ defined by the identity
\begin{align}
x^2\approx xy. \label{t01}
\end{align}
\end{pro}
\begin{proof}
From \cite[Table 1]{sr}, $\mathsf{V}(S_c(a))$ is the ai-semiring variety defined by the identities
\[
x_1x_2\approx y_1y_2,\quad x \preceq x^2.
\]
By Proposition~\ref{proSR},
we only need to show that both of these two identities are derivable from the identities \eqref{SR02}, \eqref{SR03}, \eqref{SR04}, \eqref{t01} and $\sigma_n$ for all $n \geq 1$. Indeed, we have
\[
x_1x_2 \stackrel{\eqref{t01}}\approx x_1^2 \stackrel{\eqref{SR02}}\approx x_1^3 \stackrel{\eqref{SR04}}\approx y_1^3\stackrel{\eqref{SR02}}\approx y_1^2 \stackrel{\eqref{t01}}\approx y_1y_2,
\]
which yields $x_1x_2 \approx y_1y_2$. Moreover, $x \preceq x^2$ follows directly from \eqref{SR03}.
\end{proof}

Now let $\mathcal{V}$ be a nontrivial subvariety of $\mathsf{V}(SR_6)$ distinct from $\mathsf{V}(S_c(a))$.
By Proposition~\ref{T}, $\mathcal{V}$ properly contains $\mathsf{V}(S_c(a))$.
It follows from Proposition~\ref{prot1} that $\mathcal{V}$ does not satisfy the identity \eqref{t01}.
Hence there exists a semiring $S$ in $\mathcal{V}$ with elements $a, b \in S$ such that $a^2 \neq ab$,
and so $a\neq b$.
We shall establish the following relations:
\[
ab< a^2,\, a<a^2, \, b<a^2, \, ab\neq a, \, ab\neq b, \, ab\neq a+b, \, a+b\neq a, \, a+b\neq b.
\]

First, from \eqref{SR02} and \eqref{SR04} one derives the identity
\begin{align}
y \preceq x^2. \label{SR06}
\end{align}
Applying \eqref{SR06} we obtain that $ab\leq a^2$, $a\leq a^2$, and $b\leq a^2$.
Since $a^2 \neq ab$, it follows that $ab< a^2$.
If $a=a^2$, then $ab=a^2b\stackrel{\eqref{SR04}}=a^3\stackrel{\eqref{SR02}}=a^2$, a contradiction; hence $a<a^2$.
If $b=a^2$, then $ab=aa^2=a^3\stackrel{\eqref{SR02}}=a^2$, a contradiction; thus $b<a^2$.
If $ab=a$, then $ab=ab^2\stackrel{\eqref{SR04}}=a^3\stackrel{\eqref{SR02}}=a^2$, a contradiction; therefore $ab\neq a$.
A symmetric argument shows that $ab\neq b$.
If $ab=a+b$, then
\[
ab
=ab+ab=ab+a+b\stackrel{\eqref{SR03}}=a^2+b\stackrel{\eqref{SR06}}=a^2,
\]
a contradiction; hence $ab\neq a+b$.
If $a+b=a$, then
\[
ab=(a+b)b=ab+b^2\stackrel{\eqref{SR06}}=b^2\stackrel{\eqref{SR06}}=a^2,
\]
a contradiction; thus $a+b\neq a$.
Similarly, $a+b\neq b$.

Let $\langle a, b \rangle$ denote the subalgebra of $S$ generated by $a$ and $b$.
From the relations established above, together with the identities \eqref{SR02}, \eqref{SR03}, \eqref{SR04}, and \eqref{SR06},
it follows that $\langle a, b \rangle$ consists precisely of the elements
\[
a^2,\; ab,\; a+b,\; a,\; b,
\]
with the possible exception that $a^2$ and $a+b$ may coincide; all other pairs among these five are distinct.
Set $I = \{a^2, a+b\}$.
One readily verifies that $I$ is both a multiplicative ideal and an additive order filter of $\langle a, b \rangle$.
The quotient algebra $\langle a, b \rangle / I$ is easily seen to be isomorphic to $S_c(ab)$.
Consequently, $S_c(ab)$ lies in $\mathsf{V}(S)$, and therefore $\mathsf{V}(S_c(ab))$ is a subvariety of $\mathcal{V}$.

If $\mathcal{V}\neq \mathsf{V}(S_c(ab))$, then by Corollary~\ref{cor26022310},
$\mathcal{V}$ does not satisfy the identity \eqref{26022301}.
Hence there exists a semiring $S \in \mathcal{V}$ with elements $a, b, c, d \in S$ such that
\[
ad \nleq ab + bc + cd.
\]
Our aim is to show that $\mathsf{V}(S)$ contains $SR_6$, which then implies $\mathcal{V}=\mathsf{V}(SR_6)$.

Let $\langle a, b, c, d \rangle$ denote the subalgebra of $S$ generated by $a, b, c, d$.
Then by the identities \eqref{SR02}, \eqref{SR03}, \eqref{SR04}, \eqref{SR06} and $\sigma_n$ for all $n \geq 1$,
every element of $\langle a, b, c, d \rangle$ is a finite sum of elements from the set
\[
\{a,\; b,\; c,\; d,\; ab,\; ac,\; ad,\; bc,\; bd,\; cd,\; a^2\}.
\]
In particular, $a^2$ acts as both the additive top element and the multiplicative zero.
One can verify that $\langle a, b, c, d \rangle$ contains at most $85$ elements.
Consider the following subsets of $\langle a, b, c, d \rangle$:
\[
\begin{aligned}
R_2 &= \{a,\; a+c\}, \\
R_3 &= \{ab,\; bc,\; cd,\; ab+bc,\; ab+cd,\; bc+cd,\; ab+bc+cd\}, \\
R_4 &= \{d,\; b+d\}, \\
R_5 &= \{c\}, \\
R_6 &= \{b\}.
\end{aligned}
\]
Set $R_1 = \langle a, b, c, d \rangle \setminus (R_2 \cup R_3 \cup R_4 \cup R_5 \cup R_6)$.
A direct verification (which is straightforward but somewhat lengthy) confirms that
the subsets $R_1,\dots,R_6$ are pairwise disjoint and compatible with addition and multiplication;
we omit it for brevity.

Let $\rho$ be the equivalence relation on $\langle a, b, c, d \rangle$ whose equivalence classes are precisely $R_1, \dots, R_6$.
Consequently, $\rho$ is a congruence on $\langle a, b, c, d \rangle$.
Finally, one can show that the quotient algebra
$\langle a, b, c, d \rangle / \rho = \{R_i\mid 1\leq i \leq 6\}$ is isomorphic to $SR_6$, where the isomorphism sends each congruence class $R_i$ to the element $i \in SR_6$, $1 \leq i \leq 6$.
Thus $\mathsf{V}(S)$ contains $SR_6$, and so $\mathcal{V}=\mathsf{V}(SR_6)$.

By the above arguments, we arrive at the following result, which describes the subvariety lattice of $\mathsf{V}(SR_6)$.

\begin{pro}\label{coro}
The subvariety lattice of $\mathsf{V}(SR_6)$ is a four-element chain (see Figure~$\ref{figure1}$),
where $\mathbf{T}$ denotes the trivial variety.

\begin{figure}[ht]
\centering
\setlength{\unitlength}{1.2cm}
\begin{picture}(2.5,3.2)
\thicklines
\put(1.5,0.3){\line(0,1){2.7}}
\put(1.5,0.3){\circle*{0.15}}
\put(1.5,1.2){\circle*{0.15}}
\put(1.5,2.1){\circle*{0.15}}
\put(1.5,3.0){\circle*{0.15}}
\put(1.8,0.25){\makebox[0pt][l]{$\mathbf{T}$}}
\put(1.8,1.15){\makebox[0pt][l]{$\mathsf{V}(S_c(a))$}}
\put(1.8,2.05){\makebox[0pt][l]{$\mathsf{V}(S_c(ab))$}}
\put(1.8,2.95){\makebox[0pt][l]{$\mathsf{V}(SR_6)$}}
\end{picture}
\caption{The subvariety lattice of $\mathsf{V}(SR_6)$}
\label{figure1}
\end{figure}
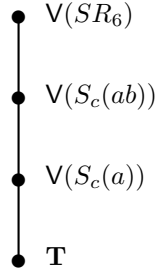
\end{pro}

By Proposition~\ref{pro26012310}, Corollary~\ref{LTN}, and Proposition~\ref{coro}, we obtain the main result of this section:

\begin{thm}\label{main}
$\mathsf{V}(SR_6)$ is a limit variety.
\end{thm}

\section{Conclusion}
The explicit limit varieties of ai-semirings constructed in \cite{rjzl} are of two types: one consists of varieties generated by flat extensions of certain groups, and the other is the variety $\mathsf{V}(S_c(abc))$. These varieties are distinct from $\mathsf{V}(SR_6)$, as they do not satisfy the inequality~\eqref{SR04}. The limit variety given in \cite{gr} is not finitely generated. Hence $\mathsf{V}(SR_6)$ provides a new example of a limit variety of ai-semirings.

Prior to this work, the eight-element algebra $S_c(abc)$ was the smallest known example of an ai-semiring generating a limit variety.
We show that the variety $\mathsf{V}(SR_6)$ is a limit variety, thereby establishing the six-element algebra $SR_6$ as the smallest known example of this kind.
This discovery not only provides a new natural example but also reduces the minimal possible size of an ai-semiring generating a limit variety from eight to six.

Results from \cite{jrz, gmrz, sr, zrc} confirm that no ai-semiring of order less than four can generate a limit variety.
Furthermore, recent studies \cite{rlzc, rlyc, yrg, yrzs} suggest that four-element ai-semirings are unlikely to possess this property, although a definitive conclusion remains open. If a smaller example exists, it would most likely be of order five.
However, with $15751$ five-element ai-semirings (see~\cite[Table~3]{emr}),
a comprehensive investigation presents a considerable challenge and offers a rich direction for future research.

The results obtained in this paper contribute to a deeper understanding of the boundary between finitely based and nonfinitely based varieties, and provide new insights into the structure of limit varieties. The question of whether a five-element ai-semiring can generate a limit variety remains an intriguing open problem for further exploration.

\qquad

\noindent
\textbf{Acknowledgements}
The authors would like to thank Zidong Gao, Jun Jiao, Chenyu Yang, and Ting Yu for their helpful discussions and contributions to this work. We also sincerely thank the anonymous referee for the extremely careful reading and for the many insightful and constructive comments and suggestions, which have significantly improved the quality of this paper.
Miaomiao Ren, Corresponding author, is supported by National Natural Science Foundation of China (12371024, 12571020).
Mengya Yue is supported by the Research Innovation Project for Postgraduates of Northwest University (CX2026052).

\bibliographystyle{amsplain}


\end{document}